\documentclass[11pt]{amsart}

\usepackage{epsfig}
\usepackage{epstopdf}
\usepackage{graphicx}
\usepackage{latexsym}
\usepackage{amsfonts,amsmath,amssymb}
\usepackage{url}
\newtheorem{theorem}{Theorem}[section]
\newtheorem{proposition}[theorem]{Proposition}
\newtheorem{lemma}[theorem]{Lemma}

\newtheorem{corollary}[theorem]{Corollary}

\theoremstyle{remark}

\newtheorem{remark}[theorem]{Remark}
\newcommand{\res}{\mathrm{Res } }

\def\cal{\mathcal}

\author[M. B\'ona]{Mikl\'os  B\'ona and Istv\'an Mez\H{o}}
\title[Vertices of a given rank]{Limiting probabilities for vertices of a given rank in rooted trees}
\address{\rm M. B\'ona, Department of Mathematics, 
University of Florida,
358 Little Hall, 
PO Box 118105, 
Gainesville, FL 32611--8105 (USA)\\
\rm I. Mez\H{o}, School of Mathematics and Statistics, 
Nanjing University of Information Science and Technology,
Nanjing, 210044, (P.R. China)
}

\date{\today}

\begin{document}

\begin{abstract}
We consider two varieties of labeled rooted trees, and the probability that a vertex chosen from all vertices of all trees
of a given size uniformly at random has a given rank. We prove that this probability converges to a limit as the tree size
goes to infinity. 
\end{abstract}

\maketitle

\section{Introduction}

Let $\cal T$ be a class of rooted labeled trees. If $v$ is a vertex of a  tree $T\in \cal T$, then let the \emph{ rank} of $v$ be the number of edges in the
shortest  path from $v$ to a leaf of $T$ that is a descendant of $v$. So leaves are of rank 0, neighbors of leaves are of rank 1, and so on. For a fixed $n$, consider all vertices of all trees in $\cal T$ that have $n$ vertices, 
and choose one vertex uniformly at random. Let $a_{n,k}$ be the probability that the chosen vertex is of rank $k$. It is then natural to ask whether the limiting probability
\[a_k = \lim_{n\rightarrow \infty} a_{n,k} \]
exists.

For one tree variety, {\em decreasing binary trees}, it has been shown \cite{protected}, \cite{janson} that these limits $a_k$ exist, and the values of $a_k$ were explicitly computed in  \cite{bona-pittel} for $k\leq 6$. Recursive trees
are discussed in \cite{holmgren}.  However, the methods that were successful for these trees are often unsuccessful for other tree varieties if $k>1$. This is because many of the relevant differential equations cannot be solved, or even, explicitly
stated, caused by the fact that many of the relevant functions lack an elementary antiderivatives. We will explain this phenomenon in Section \ref{nonplane12}.

This raises the intriguing question whether we can prove that $a_k$ exists  for some of these tree varieties, {\em even though we cannot explicitly compute its value}. In this paper we will answer that question in the affirmative for two
labeled tree
varieties, {\em non-plane 1-2 trees}, and {\em plane 1-2 trees}. For $k=0$ and $k=1$, we are able to compute the
exact values of $a_k$. 

\section{Non-plane 1-2 trees} \label{nonplane12}

Our first example is the class of {\em labeled non-plane 1-2 trees}. In such trees, every non-leaf vertex has at most two children, the vertices are bijectively labeled with the elements of $[n]=\{1,2,\cdots ,n\}$ so that the label of each vertex is less than that of its parent, and the set of children of any given vertex is unordered. See Figure \ref{fig:fivetrees} for the five non-plane 1-2 trees on vertex set $[4]$. In this section, when we write {\em tree}, we will always mean a labeled
non-plane 1-2 tree on vertex set $[n]$. 

\begin{figure}
 \begin{center}
  \includegraphics[width=70mm]{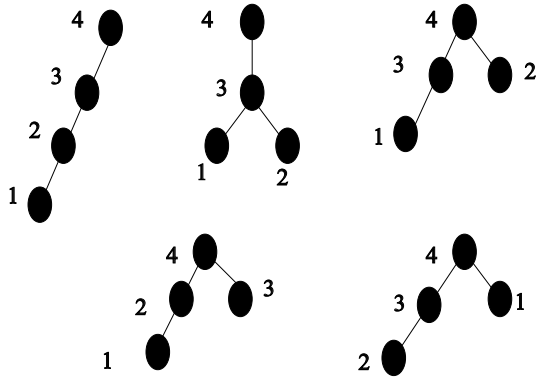}
  \caption{The five rooted non-plane 1-2 trees on vertex set $[4]$. }
  \label{fig:fivetrees}
 \end{center}
\end{figure}

It is well-known \cite{FL} that the number  of labeled non-plane 1-2 trees on vertex
set $[n]$ is the $n$th Euler number $E_n$, and that the identity
\begin{equation} 
\label{euler} E(z)=\sum_{n\geq 0} E_n \frac{z^n}{n!} = \sec z + \tan z
\end{equation}
holds, where we set $E_0=1$.

The first values of $E_n$ are as follows:

\begin{center}
\begin{tabular}{c|c|c|c|c|c|c|c|c|c|c|c|c}
$n$&0&1&2&3&4&5&6&7&8&9&10\\
$E_n$&1&1&1&2&5&16&61&272&1385&7936&50521
\end{tabular}
\end{center}

It follows from (\ref{euler}) that $E(z)$ has two singularities of smallest modulus, at $z=\pi/2$ and at $z=-\pi/2$. Therefore, the exponential order of growth of the Euler numbers is $2/\pi$. In order to find the growth rate of the Euler numbers more precisely, note that  at both of  these singularities, we can find the residue of $E(z)$ by the following well-known formula.

\begin{proposition} \label{residues}
Let $H(z)=f(z)/g(z)$ be a function so that $f(z)$ and $g(z)$ are analytic
functions at $z_0$, and $f(z_0)\neq 0$, while $g(z)=0$ and $g'(z)\neq 0$. Then 
\[\res H(z) \mid_{z_0} =\frac{f(z_0)}{g'(z_0)}.\]
\end{proposition}

We can apply Proposition \ref{residues} to $E(z)$ if we note that
$E(z)=\frac{1+\sin z}{\cos z}$. Then  Proposition \ref{residues} implies
that $\res E(z)\Bigr|_{\pi/2}=\frac{2}{-1}=-2$, and $\res E(z) \mid_{-\pi/2}=\frac{0}{1}=0$. 

Now observe that
\begin{eqnarray} \frac{R}{z-a} & = & \frac{R}{-a} \cdot \frac{1}{1-\frac{z}{a}}
\\
& = &  \frac{R}{-a} \sum_{n\geq 0} \frac{z^n}{a^n}.
\end{eqnarray} 

Applying this observation  to $E(z)$ with $a=\pi/2$ and $R=-2$, we get that the dominant term of $E(z)$ is of the form $\frac{4}{\pi}\sum_{n\geq 0} z^n (2/\pi)^n$, so 
\begin{equation}
\label{eulerprecise}
\frac{E_n}{n!} \sim \frac{4}{\pi} \cdot  \left(\frac{2}{\pi} \right)^n .
\end{equation}

Now we proceed to determine $a_0$ and $a_1$. For these small values of $k$, we can explicitly determine $a_k$, 
but we will also see why the same approach fails for larger values of $k$. 

\subsection{Leaves}
Now let $A_{0,n}$ denote the total number of {\em leaves} in all non-plane 1-2 trees on vertex set $[n]$. Then $A_{0,0}=0$, $A_{0,1}=A_{0,2}=1$, while $A_{0,3}=3$, $A_{0,4}=9$, $A_{0,5}=35$, and $A_{0,6}=155$. 

\begin{theorem} Let $A_0(z)=\sum_{n\geq 0}A_{0,n} \frac{z^n}{n!}$. Then 
\begin{equation} \label{formofb} A_0(z)=\frac{z-1+\cos z}{1-\sin z}.
\end{equation}
\end{theorem}

\begin{proof}
Let $(v,T)$ be an ordered pair in which $T$ is a non-plane 1-2 tree on vertex set $[n]$ and $v$ is a leaf of $T$. Then $A_0(z)$ is the exponential generating function counting such pairs.  Let us first assume that $n>1$, and let us remove the root of $T$.  On the one hand, this leaves a structure that is counted by $A_0'(z)$. On the other hand, this leaves an ordered pair consisting of a non-plane 1-2 tree with a leaf marked, and a non-plane 1-2 tree. By the Product formula of exponential generating functions (see \cite[Theorem 8.21]{WalkThroughComb}), such ordered pairs are counted by the generating function $A_0(z)E(z)$. Finally, if $n=1$, then 
no such ordered pair is formed, while $A_0'(z)$ has constant term 1.  This leads to the linear differential equation
\begin{equation}
A_0'(z)=E(z)A_0(z) +1,\label{deA0}
\end{equation}
with the initial condition $A_0(0)=0$. Solving this equation we get formula (\ref{formofb}) for $A_0(z)$. 
\end{proof}

In order to determine the growth rate of the numbers $A_{0,n}$, we will need the following lemma, which is an enhanced version of Proposition \ref{residues}. 

\begin{lemma} \label{doubleres}
Let $H(z)=\frac{f(z)}{g(z)}$ be a function so that $f$ and $g$ are analytic functions at $z_0$, and $f(z_0)\neq 0$, while $g(z_0)=g'(z_0)=0$, and $g''(z)\neq 0$. Then 
\[H(z)=\frac{2f(z_0)}{g''(z_0)} \cdot \frac{1}{(z-z_0)^2} + \frac{h_{-1}}{z-z_0}
+h_0+\cdots.\]
\end{lemma}

\begin{proof}
The conditions directly imply that $g$ has a double root, and hence $H$ has a pole of order two, at $z_0$. In order to find the coefficient that belongs to that pole, let 
$g(z)=q(z)(z-z_0)^2$. Now differentiate both sides twice with respect to $z$, to get
\[g''(z)=q''(z)(z-z_0)^2 + 4q'(z)(z-z_0) + 2q(z).\]
Setting $z=z_0$, we get
\begin{equation}\label{fandg} g''(z_0)=2q(z_0).
\end{equation}
By our definitions, in a neighborhood of $z_0$, the function $H(z)$ behaves like
\[\frac{f(z)}{q(z)(z-z_0)^2},\] and our claim follows by (\ref{fandg}).
\end{proof}

\begin{theorem} \label{rank0}
The equality 
\[a_0=\lim_{n\rightarrow \infty} \frac{A_{0,n}}{nE_n} = 1-\frac{2}{\pi} \approx 0.3633802278\]
holds. In other words, for large $n$, the average non-plane 1-2 tree on vertex 
set $[n]$ has about $(n+1)\cdot ( 1-\frac{2}{\pi})$ leaves. 
\end{theorem}

\begin{proof}
Note that $A_{0}(z)$ has a unique singularity of smallest modulus, at $z=\pi/2$, hence the exponential growth rate of  its coefficients is $2/\pi$. Also note that at that point, the denominator of $A_0(z)$ has a double root. Therefore, Lemma \ref{doubleres} applies and we get that the coefficient of the $(z-\pi/2)^{-2}$ term in the Laurent series of $A_0(z)$ is 
\[2\cdot \frac{(\pi)/2 -1 +\cos (\pi/2)}{\sin(\pi/2)}=\pi-2.\]
Now observe that 
\begin{eqnarray} \frac{D}{(z-a)^2} & = & \frac{D}{a^2} \cdot
 \frac{1}{(1-\frac{z}{a})^2} \\
 & = & \frac{D}{a^2} \cdot \sum_{n\geq 0}(n+1)\frac{z^n}{a^n}.
\end{eqnarray}
Applying this to the dominant term of $A_0(z)$ with $D=\pi-2$ and $a=\pi/2$, we get that
\begin{equation}\label{leafprecise}
\frac{A_{0,n}}{n!} \sim (n+1) (\pi -2) \cdot \left(\frac{2}{\pi}\right)^{n+2}.
\end{equation}
The proof of our claim is now immediate by comparing formulas  (\ref{leafprecise}) and (\ref{eulerprecise}).
\end{proof}

It is worth pointing out that $\int A_0(z)= 1-(1-z)(\tan z+ \sec z)$, which implies the identity $A_{0,n}=(n+1)E_{n} -E_{n+1}$. See sequence A034428 in the Online Encyclopedia of Integer Sequences \cite{sloane}.

\subsection{Vertices of rank 1} \label{sectionrank1}
Let $A_{1,n}$ be the total number of vertices in all non-plane 1-2 trees on vertex set $[n]$ that are of rank 1. Note that 
such vertices are neighbors of a leaf. If $n>1$, then each leaf has exactly one rank-1 vertex as a neighbor, while some rank-1 vertices have not only one, but two leaves as neighbors.

The first few members of the sequence $A_{1,n}$ are $A_{1,0}=0$, $A_{1,1}=0$, $A_{1,2}=1$, $A_{1,3}=2$, $A_{1,4}=8$, $A_{1,5}=30$, and  $A_{1,6}=135$.

Let $A_1(z)=\sum_{n\geq 0}A_{1,n}\frac{z^n}{n!}$. Let $(v,T)$ be an ordered pair in which $T$ is a non-plane 1-2 tree on  vertex set $[n]$ and $v$ is a vertex of $T$ that is of rank 1.  If $v$ is not the root of $T$, then removing the root of $T$ decomposes $(v,T)$ into two structures, one of which is again a non-plane 1-2 tree with a vertex of rank 1 marked, and the other one of which is simply a non-plane 1-2 tree.  If $n>1$, and $v$ is the root of $T$, then removing $v$, we get two structures, one of which is a leaf, and the other one is a non-plane 1-2 tree. These two structures are distinguishable unless the original tree had three vertices, and its root had two children. That tree contributed $z^3/6$ to the generating function $A_1(z)$, but that contribution was counted twice. This leads to the linear differential equation
\begin{equation}\label{formfora1} A_1'(z)= A_1(z)\cdot E(z) + zE(z)-\frac{z^2}{2!},\end{equation}
with the initial condition $A_1(0)=0$. 

Solving this equation we get
\begin{equation}
\label{formofv}
A_1(z)= \frac{1}{6} \cdot
 \frac{12z\sin z +12\cos z -12 -3z^2\cos z -z^3}{1-\sin z}.
\end{equation}
The above formula for $A_1(z)$ shows that $A_1(z)$ has a unique singularity  of smallest modulus, at $z=\pi/2$. Therefore, the exponential growth rate of the coefficients of $A_1(z)$ is $2/\pi$. At $z=\pi/2$, the power series $A_1(z)$ has a pole of order two, since the denominator has a double root at that point, while the numerator is non-zero there.

Therefore, we can apply Lemma \ref{doubleres} with
 $f(z)=12z\sin z +12\cos z -12 -3z^2\cos z -z^3$ and $g(z)=6(1-\sin z)$. 
Then $f(\pi/2)=6\pi-\frac{\pi^3}{8} -12$, while $g''(\pi/2)=6$.
Hence Lemma \ref{doubleres} shows that the dominant term of $A_1(z)$ is
of the form 
\[\frac{2\pi - \frac{\pi^3}{24} -4}{(z-\frac{\pi}{2})^2}=
\frac{2\pi - \frac{\pi^3}{24} -4}{(\pi/2)^2} \cdot
\sum_{n\geq 0}(n+1)\frac{z^n}{(\pi/2)^n}.\]
This implies that
\[\frac{A_{1,n}}{n!} \sim (n+1)\cdot \frac{2\pi - 
\frac{\pi^3}{24} -4}{(\frac{\pi}{2})^{n+2}}.\]
Comparing this to (\ref{eulerprecise}), we get the following theorem. 

\begin{theorem} \label{rank1} The equality
\[a_1=\lim_{n\rightarrow \infty} \frac{A_{1,n}}{(n+1)E_n} = 2-\frac{\pi^2}{24}-
\frac{4}{\pi}\approx 0.3155269391\]
holds.
\end{theorem} 

\begin{remark} \label{nointegral}  Note that it directly follows from the argument we used to prove (\ref{formfora1}) that if $R_1(z)$ is the exponential generating function for the number of non-plane 1-2 trees on vertex set $[n]$ whose root is of rank 1, then
\[R_1'(z)= zE(z) -\frac{z^2}{2}=z\tan z + z\sec z -\frac{z^2}{2}.\]
Unfortunately, this closed form for $R_1'(z)$ does not lead to a closed form for $R_1(z)$, since $R_1(z)$ does not have an elementary antiderivative. 
\end{remark}

\subsection{Vertices of rank $\ge2$} \label{higher-nonplane}

The methods that we used to enumerate vertices of rank 0 and rank 1 will fail for vertices of higher rank, because we are not able to solve the linear differential equations analogous to (\ref{formfora1}), since the relevant functions have no elementary antiderivatives. Remark \ref{nointegral} shows how early these kind of problems start occuring; we are not
even able to state the equation analogous to (\ref{formfora1}) in an explicit form. 

Therefore, we apply a new method to prove that the limit $a_k=\lim_{n\rightarrow \infty}a_{n,k}$ exists. We will then be able to  approximate $a_k$
from above and below.  Our first simple notion is the following. 
 Each vertex of a tree is the root of a unique subtree, which we will
call {\em the subtree of the vertex}. In other words, the subtree of a vertex $v$ consists of all descendants of $v$, including $v$ itself. The subtrees of leaves consist of one vertex only.

For a fixed positive integer $r$, let $V_{n,r}$ be the probability that a randomly selected vertex in a random
non-plane 1-2 tree of size $n$ is the root of a subtree of size exactly $r$. For instance, if $n=3$, then $V_{n,1}=1/2$, 
$V_{n,2}=1/6$, and $V_{n,3}=1/3$. 

Vertices of a given subtree size are much easier to enumerate than vertices of a given rank, because the number of
ways in which a vertex can have a subtree of size $r$ is a fixed number, namely the Euler number $E_n$, once the
set of labels in that subtree is selected. 

\begin{proposition} \label{fixedr}
For all positive integers $r$, the limit
\[v_r = \lim_{n\rightarrow \infty} V_{n,r} \] exists.
\end{proposition}

\begin{proof} Let $V_r(z)=\sum_{n\geq 0} V_{n,r}z^n/n!$. Then by the Product formula, we have
\[V_r'(z)=V_r(z)E(z) + f_r'(z),\] where $f_r(z)$ is the generating function for the number of trees in which {\em the root}
has a "subtree" of size $r$. That is, $f_r(z)=E_r z^r /r!$, so $f_r'(z)=E_rz^{r-1}/(r-1)!$, and the last displayed equation 
becomes 
 \begin{equation} \label{treesize} V_r'(z)=V_r(z)E(z) + E_r\frac{z^{r-1}}{(r-1)!} .\end{equation}
This is a first order linear differential equation with initial condition $V_r(0)=0$. Its solution is 
\[V_r(z)=\frac{E_r}{(r-1)!} \frac{\int z^r(1-\sin z) \  dz }{1-\sin z} +\frac{C}{1-\sin z},\] where $C$ is to be selected 
so that $V_k(0)=0$ holds.
 
The integral in the numerator can be explicitly computed using  the well-known formula
\[\int z^r \sin z \ dz  = \cos z \sum_{i=0}^{r/2} (-1)^{i+1}z^{r-2i} \frac{r!}{(r-2i)!} \] \[+ \sin z\sum_{i=0}^{(r-1)/2}
(-1)^i z^{r-2i-1} \frac{r!}{(r-2i-1)!}.\] 
This means that $V_r(z)$ has a unique singularity of smallest modulus (a double pole) at $z_0=\pi/2$. The rest of the argument uses Lemma \ref{doubleres} at $z_0=\pi/2$ to determine $v_r$, in the same way as we did in the proofs of 
Theorems \ref{rank0} and \ref{rank1}.
\end{proof}

\begin{remark} \label{polynomial} 
 Note that we are able to explicitly solve the linear differential equation (\ref{treesize}) because 
its "correction term", that is, the summand that does not contain $V'_r(z)$ or $V_r(z)$, is a {\em polynomial}. 
The same argument used here would work for any polynomial instead of $f_r'(z)E_rz^{r-1}/(r-1)!$.  \end{remark}

Proposition \ref{fixedr} shows that the limit $v_r$ exists for every fixed positive integer $r$. 
As the $v_r$ are all positive real numbers, and $\sum_r v_r\leq 1$, the sum  $\sum_{r=1}^\infty v_r $ is convergent. However,
{\em what} is the value of that sum? The exact formulas we obtain for each $v_r$ from Proposition \ref{fixedr} are too complicated to be useful for the computation of that sum. Note that it is not true in every tree variety that the
analogously defined sum is equal to 1. A simple counterexample is the family of rooted trees in which every
non-leaf vertex has exactly one child. However, for our non-plane 1-2 trees, the sum turns out to be 1, although
not in a trivial way. This is the content of the following Theorem that has been conjectured by the present authors, and has recently been proved by  Svante Janson \cite{jansonpers}. 

\begin{theorem} \label{tight}
The equality \[\sum_{r=1}^\infty v_r=1 \]
holds.
\end{theorem}

In order to prove Theorem \ref{tight}, we first need the following fact. 

\begin{proposition} \label{every} For all $n$, the expected number of leaves  in a  random  non-plane 1-2 tree on vertex set $[n]$ is
 is at least $1/4$.   \end{proposition}

Note that in Theorem \ref{rank0}, we proved that much more is true for large $n$. However, the statement of Proposition
\ref{every} is true for {\em every} $n$.  With a little bit of additional work, it is possible to prove that in fact, 
the expected number of leaves is at least $n/3$ for every $n$, but the weaker claim of Proposition \ref{every} suffices
for our purposes. 

\begin{proof} Every such tree contains exactly one more leaves than vertices with two children. Therefore, it suffices to show 
that the expected number of vertices with one child is not more than $n/2$. We prove this by induction on $n$. Let $M_n$ be the 
expected number of vertices with one child  in a random   tree on vertex set $[n]$. Then $M_1=0$, $M_2=1$, and
$M_3=1$, so the statement holds if $n\leq 3$. Now let us assume that $n>3$.  Let $p_n$ be the probability that the {\em root} of a non-plane 1-2 tree on $[n]$ has exactly one child. Then 
by conditional expectations, we have 
\[M_n  \leq (1-p_n) \cdot \frac{n-1}{2} +p_n \left(\frac{n-1}{2}+1\right)=\frac{n-1}{2}+p_n\leq \frac{n}{2}.\]
The last inequality follows, since $p_n=E_{n-1}/E_{n}$, the ratio of two consecutive Euler numbers, and the Euler
numbers are known to be log-convex \cite{Liu}. So the sequence of the numbers $p_n$ is decreasing. As $p_3=1/2$, it follows that $p_n\leq 1/2$ if
$n\geq 3$.  
\end{proof}

Let $Z_n$ be a random variable defined on the set of all vertices of all
 non-plane 1-2 trees on vertex set $[n]$, so that $Z_n(v)$ is the size of the subtree rooted at $v$. 

\begin{lemma} \label{estimate} For all $n$, the inequality $E(\sqrt{Z_n}) \leq  100$ holds.
\end{lemma}

\begin{proof} We will use strong induction to  prove the stronger inequality 
\begin{equation} E(\sqrt{Z_n}) \leq 100 -\frac{90}{\sqrt{n}}. \end{equation}
This inequality clearly holds for $n=1$. Now let us assume that it holds for all positive integers $r<n$ and prove it for $n$. 

Let $T$ be a tree of size $n$, selected uniformly at random. Let $v$ be a vertex of $T$, selected in the same way. 

Then by Proposition \ref{every}, there is an at least $1/4$ chance of $v$ being a leaf. There is an  $1/n$ chance of $v$ being the root. There is a less than $3/4$ chance of $v$ being another vertex, in which case the induction hypothesis
applies to the subtree of $v$, with some $r<n$ playing the role of $n$. Therefore,  
\[E(\sqrt{Z_n}) \leq \frac{3}{4}  \cdot \left (100-\frac{90}{\sqrt{n}} \right ) + \frac{1}{4}+ \frac{1}{n}\cdot\sqrt{ n} \]
\[=75.25-\frac{66.5}{\sqrt{n}}\]
\[\leq   100 -\frac{90}{\sqrt{n}},\]
where the  last estimate holds as $n\geq 1$. 
\end{proof}

Now we are ready to prove Theorem \ref{tight}. 
\begin{proof}(of Theorem \ref{tight})
By Lemma \ref{estimate} and by Markov's inequality, we know that for all positive constants $C$, and for all $n$, we have
\[ Pr(\sqrt{Z_n}> 100C ) \leq 1/C, \]
so
\begin{equation} \label{markovineq} Pr(Z_n>10000C^2 ) \leq 1/C.\end{equation}

Let us now assume that $\sum_{k=1}^\infty v_k = \alpha < 1$. That means that for all $N$, the inequality 
$\sum_{k=1}^N v_k < \alpha $ holds. In other words, if $n$ is large enough, then in an average non-plane 
1-2 tree on $[n]$, there are at least $(1-\alpha)n/2$ vertices whose subtree is of size more than $N$.  That is, 
\begin{equation} \label{largeenough} Pr(Z_n>N) \geq \frac{1-\alpha}{2}.\end{equation}
 Now select $C$ to be a positive integer so that $1/C < (1-\alpha)/2$, then select $N=10000C^2$.
Then inequality (\ref{markovineq}) forces \[Pr(Z_n>N) \leq 1/C <(1-\alpha)/2\] for all $n$, while
inequality (\ref{largeenough}) forces $Pr(Z_n>N)\geq (1-\alpha)/2$ for $n$ sufficiently large, which is clearly a contradiction.
\end{proof}

The following is an obvious corollary of Theorem \ref{tight} that we will need soon.

\begin{corollary} \label{ucor}
Let $U_{n,r}=1-\sum_{i=1}^r V_{n,i}$ be the probability that a random vertex of a random non-plane 1-2 tree has a 
subtree of size larger than $r$. Then clearly, 
\[u_r:=\lim_{n\rightarrow \infty} U_{n,r}= 1-\sum_{i=1}^r  v_i. \]
Furthermore, and this is where Theorem \ref{tight} is needed, 
\begin{equation} \label{ulimit} \lim_{r\rightarrow \infty} u_r = 1 - \lim_{r\rightarrow \infty} \sum_{i=1}^r  v_i=0.
\end{equation}
\end{corollary}

We now return to our main goal, that is, to proving that the limit  $a_k=\lim_{n\rightarrow \infty}a_{n,k}$ exists.
For the rest of this section, we fix the rank $k$ of vertices we are studying, and, to alleviate notation, we do {\em not
add} the index $k$ to all parameters related to these vertices. 

Our main idea is the following. The set $\mathcal R_k$ of all vertices of all trees of size $n$ {\em contains} the
set $\cup_{i=1}^r \mathcal W_{n,i}$, where $\mathcal W_{n,i}$ is the set of all vertices of all trees of size $n$ that
are of rank $k$ and have a subtree of size $i$. On the other hand, $\mathcal R_k$ {\em is contained} in the set
$\left(\cup_{i=1}^r \mathcal W_{n,i} \right) \cup \left( \cup_{i>r} \mathcal V_{n,i} \right)$, where $ \mathcal V_{n,i} $
is the set of all vertices in all trees of size $n$ whose subtree is of size $i$ (but are of any rank). 

Let $W_{n,i}$ be the probability that a random vertex of  a random tree  of size $n$  is of rank $k$ and is the root of 
a subtree  of size $i$.  Let 
\[w_i=\lim_{n\rightarrow \infty} W_{n,i} .\]
The limits $w_i$ exist, because the exponential generating functions of the numbers $W_{n,i}$ satisfy a linear
differential equation like (\ref{treesize}), and, as explained in Remark \ref{polynomial}, we can explicitly solve those differential equations, since
their "correction term" is a polynomial. Indeed, there are only a finite number of ways that a subtree of a vertex can 
be of rank $k$ and have a subtree of size $i$, once the set of labels going into that subtree is selected. 

As the $w_i$ are positive real numbers, and for all $r$, the inequality $\sum_{i=1}^r w_i \leq 1$ holds, the sum 
\[w=\sum_{i=1}^{\infty} w_i \] exists. 

Now we are ready to state and prove our main theorem. 
\begin{theorem} \label{maintheorem}
For all positive integers $k$, the limit \[a_k:= \lim_{n\rightarrow \infty} a_{n,k} \] exists. Furthermore, 
 \[a_k= w.\]
\end{theorem}

\begin{proof} 
First notice that for all $n$ and $r$, the inequality 
\[\sum_{i=1}^r W_{n,i}\leq a_{n,k} \] holds, since the left-hand side is the probability of a random vertex
having a more restrictive property (rank $k$, subtree size at most $r$) than the property represented on the
right-hand side (rank $k$). 
Therefore, 
\begin{equation} \label{lowerineq} \sum_{i=1}^r w_r\leq \liminf_{n} a_{n,k} , \end{equation}
and so 
\begin{equation} \label{liminf} w \leq \liminf_n a_{n,k} .\end{equation}

Now notice that for all $n$ and $r$, the inequality
 \[a_{n,k} \leq \sum_{i=1}^r W_{n,i}+ \sum_{i>r} V_{n,i} = \left( \sum_{i=1}^r W_{n,i} \right )+ U_{n,r}   \]
holds. Indeed, the right-hand side is the probability of a random vertex being of rank $k$ and having a subtree of size
at most $r$, or simply having a subtree of size more than $r$ (and any rank). A particular way of this occuring is when 
the random vertex is of rank $k$, which is the event whose probability is represented on the left-hand side.

This implies that for all $r$, we have
\begin{equation} \label{upperineq}  \limsup_n a_{n,k} \leq \sum_{i=1}^r w_i + \left(1-\sum_{i=1}^r v_i\right).
\end{equation}
As $r$ goes to infinity, the first sum on the right-hand side goes to $w$, while the second sum goes to 0, as we saw
in Corollary \ref{ucor}.
This proves that  
\begin{equation} \label{limsup} \limsup_n a_{n,k} \leq w.\end{equation}

Comparing inequalities (\ref{liminf}) and (\ref{limsup}), we see that 
\[\limsup_n a_{n,k} \leq w \leq  \liminf_n a_{n,k} ,\] completing the proof of the theorem. 
\end{proof}

For numerical approximations, one can use the following corollary, which is an immediate consequence of Theorem
\ref{maintheorem} that we have just proved, and inequalities (\ref{lowerineq}) and (\ref{upperineq}) that we have
used in the proof of that theorem. 

\begin{corollary} \label{numerical} For all $r$, the chain of inequalities
\[\sum_{i=1}^r w_i  \leq a_k \leq \sum_{i=1}^r w_i + \left ( 1 - \sum_{i=1}^r v_i \right ) 
\]
holds. 
\end{corollary}

\section{Plane 1-2 trees}
The next tree class we study is the class of {\em plane} 1-2 trees on vertex set $[n]$. These are similar to the trees
of the previous section, except that now the order of the children of  each vertex matters. See Figure \ref{fig:threetrees} 
for an illustration. We denote the number of such trees on $[n]$  by $b_n$. Our goal is to show that Theorem \ref{maintheorem} can be proved for these trees as
well. Most steps are similar to what we saw in Section \ref{nonplane12}, but there will be one step that requires
a separate argument. 

\begin{figure}
 \begin{center}
  \includegraphics[width=70mm]{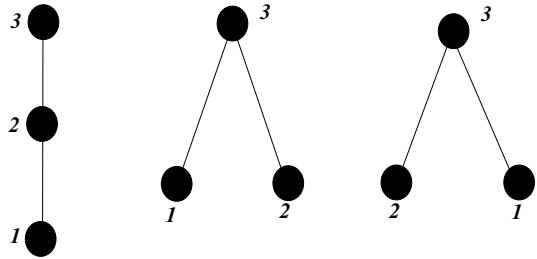}
  \caption{The three rooted plane 1-2 trees on vertex set $[3]$. }
  \label{fig:threetrees}
 \end{center}
\end{figure}

The first few values of the sequence $b_n$, starting with $b_1$, are  1, 1, 3, 9, and 39. This is sequence 
A080635 in OEIS \cite{sloane}.  Setting $b_0=1$,   the exponential generating function
\[B(z)=\sum_{n=0}^\infty b_n\frac{z^n}{n!}\]
satisfies the differential equation
\[B'(z)=1-B(z)+B^2(z).\]
Solving this equation yields
\[B(z)=\frac{1}{2} +  \frac{\sqrt{3}}{2} \tan \left(\frac{\sqrt{3}}{2}z +\frac{\pi }{6}\right).\]
The power series form of $B(z)$ leads to the asymptotic formula
\begin{equation}
\frac{b_n}{n!}\sim\frac{3^{3/2}}{2\pi}\left(\frac{3^{3/2}}{2\pi}\right)^n.\label{bnasymp}
\end{equation}

Just as it was the case for non-plane trees, we can determine the values of $a_0$ and $a_1$ for plane 1-2 trees as
well. 
\subsection{Leaves} \label{leaves-plane}

We can find the ratio of leaves among all vertices in a way that is analogous to that for non-plane 1-2 trees. 
\begin{theorem}The exponential generating function of the numbers of leafs in decreasing plane 1-2 trees is
\[B_0(z)=\sum_{n=0}^\infty b_{0,n}\frac{z^n}{n!}=\frac{6 z+\sqrt{3} \sin \left(\sqrt{3} z\right)+3 \cos \left(\sqrt{3} z\right)-3}{-3 \sqrt{3} \sin \left(\sqrt{3} z\right)+3 \cos \left(\sqrt{3} z\right)+6},\]
which satisfies the differential equation
\[B_0'(z)=2B_0(z)(B(z)-1)+B_0(z)+1.\]
\end{theorem}

\begin{proof} Just like in our proofs for analogous results in Section \ref{nonplane12}, we count ordered pairs
$(v,T)$, where $v$ is a leaf of the tree $T$. Let us remove the root of $T$. Then there are two cases, namely, 
either the removed root was $v$, or it was not. 

If $v$ is not the root, and we got two trees, one with the marked vertex, then the Product formula yields 
the generating function $2B_0(z)(B(z)-1)$, as the order of the components matters.  

If  $v$ is not the root, and the root has only one child, then removing the root, we got only one tree, with a marked vertex, which contributes the generating function $B_0(z)$.

Finally, if  $v$ is the root, the only possible tree is the one-point graph. The removal of that root leads to the empty graph, represented by $1=\left(\frac{z^1}{1!}\right)'$ in the differential equation.
\end{proof}

By the generating function we can determine the first values of $b_{0,n}$:
\begin{center}
\begin{tabular}{c|c|c|c|c|c|c|c|c|c|c|c}
$n$&0&1&2&3&4&5&6&7&8&9&10\\
$b_{0,n}$&0&1&1&5&17&93&513&3\,477&25\,569&212\,733&1\,929\,393
\end{tabular}
\end{center}

A simple application of Lemma \ref{doubleres} yields that
\[\frac{b_{0,n}}{(n+1)n!}\sim\frac{\frac{4\pi}{9\sqrt3}-\frac13}{\left(\frac{2\pi}{3\sqrt3}\right)^{n+2}},\]
and, by recalling \eqref{bnasymp},
\[\lim_{n\to\infty}\frac{b_{0,n}}{(n+1)b_n}=\frac23-\frac{\sqrt3}{2\pi}\approx0.391.\]

\subsection{Neighbors of leaves} \label{neighbors-plane}
This case is similar to that of leaves, with some subtle differences. Let $b_{1,n}$ denote the number of all vertices of rank
1 in all trees on vertex set $[n]$, and let $B_{1}(z)$ be the exponential generating function of these numbers.  Let us count ordered pairs $(v,T)$, where $v$ is a vertex of rank 1 in a tree $T$. Let us remove the root of $T$. The case when $v$ is not the root, is the same as in Section \ref{leaves-plane}, contributing  the term $2B_1(z)(B(z)-1)+B_1(z)$. When $v$ is the root, then removing it we obtain a leaf and a tree. If this tree is not empty, we must distinguish whether it was on the left or right hand side, so we must add a term $2z(B(z)-1)$. If, in turn, the subtree is empty, we must add the term representing the path of length one, that is $\left(\frac{z^2}{2!}\right)'=z$. Finally, we must realize that the terms $2B_1(z)(B(z)-1)$ and $2z(A(z)-1)$ both contain the two trees on three points where the root has two children. Therefore we must subtract $2\left(\frac{z^3}{3!}\right)'=z^2$.  This proves the following. 

\begin{theorem}
The generating function $B_1(z)$ satisfies the differential equation
\[B_1'(z)=2B_1(z)(B(z)-1)+B_1(z)+2z(B(z)-1)+z-z^2\quad(B_1(0)=0.\]
Therefore, 
\[B_1(z)=\frac{6 z^3+\sqrt{3} \left(3 z^2-15 z-5\right) \sin \left(\sqrt{3} z\right)+3 \left(3 z^2+5 z-5\right) \cos \left(\sqrt{3} z\right)+15}{9 \left(\sqrt{3} \sin \left(\sqrt{3} z\right)-\cos \left(\sqrt{3} z\right)-2\right)}.\]
\end{theorem}

The first values of $b_{1,n}$ are as follows.

\begin{center}
\begin{tabular}{c|c|c|c|c|c|c|c|c|c|c|c}
$n$&0&1&2&3&4&5&6&7&8&9&10\\
$b_{1,n}$&0&0&1&3&15&75&435&2883&21\,447&177\,435&1\,613\,835
\end{tabular}
\end{center}

The asymptotic expression for the total number of rank one vertices can be found easily:
\[\frac{b_{1,n}}{(n+1)n!}\sim\frac{540 \sqrt{3} \pi -16 \sqrt{3} \pi ^3-1215}{2187\left(\frac{2\pi}{3\sqrt3}\right)^{n+2}},\]
and
\[\lim_{n\to\infty}\frac{b_{1,n}}{(n+1)b_n}=\frac{10}{9}-\frac{5}{2\sqrt3\pi}-\frac{8\pi^2}{243}\approx 0.3267.\]

\subsection{Vertices of higher rank} If we try to apply the method of Sections \ref{leaves-plane} and \ref{neighbors-plane} for vertices of rank $k$, for $k\geq 2$, we 
fail, because yet again, the relevant generating functions will not have elementary antiderivatives. However, the 
method that we used in Section \ref{higher-nonplane} to prove that the limits $a_k$ exist will work again, as we will show. 

Let us define the limits $v_r$ and $w_r$ exactly as we did  in Section \ref{higher-nonplane}, except that now the trees are
plane. If we try to follow the argument of the non-plane case, we see that the first step towards proving the existence
of $v_r$ and $w_r$  is to show that we can explicitly solve the
linear differential equation
\[f'(z)=2f(z)(B(z)-1) +f(z) +P(z), \]
where $P(z)$ is a {\em polynomial} function. Indeed, we get differential equations of the above kind when we 
attempt to find the probabilities $V(n,r)$ or $W(n,r)$. 

Bringing the above differential equation to standard form, we get
\begin{equation} \label{standard}  f'(z)+(1-2B(z))f(z) = P(z).\end{equation}

In order to solve  (\ref{standard}), we multiply both sides by the integrating factor
\[Q(z)=\exp \left (\int (1-2B(z)) \ dz \right)=\frac{1}{2}+\frac{\cos \left (\sqrt{3}z \right )}{4} -\frac{\sqrt{3}\sin \left (\sqrt{3}z \right )}{4}.\]

Multiplying both sides of (\ref{standard}) by $Q(z)$, we get the equation 
\[ (f(z)Q(z))' = Q(z)P(z),\]
which we can explicitly solve as long as we can integrate $Q(z)P(z)$. In the present case, we can certainly do that, since 
$P(z)$ is a polynomial function of $z$, hence it is a polynomial function of $\sqrt{3}z$ as well, so a substitution $t=\sqrt{3}z$ will result in a function consiting of the sums of summands in the form $t^m\sin t$ and $t^n \cos t$. 
In the end, we obtain 
\begin{equation} \label{generalf} f(z)=\frac{\int Q(z)P(z)}{Q(z)}=\frac{K(z)}{Q(z)}, \end{equation}  a meromorphic function. The asymptotic behavior of meromorphic
functions is well understood. See Theorem IV.10 of \cite{FL} for the most important results. In our case, 
the numerator $K(z)$ of $f(z)$ in (\ref{generalf}) is an entire function, while the denominator has a zero
at $z=2\sqrt{3}\pi /9$ that has multiplicity two. Therefore, the asymptotics of the coefficients of $f(z)$ can be computed
using Lemma \ref{doubleres}. Therefore, the existence of $v_r$ and $w_r$ can also be proved in the same way as it was
in Section \ref{nonplane12} for non-plane 1-2 trees. 

The next step is to prove Theorem \ref{tight}, that is, the equality $\sum_{r=1}^{\infty}v_rs=1$ for plane 1-2 trees.
There is one step in that proof that needs an argument that is different from its non-plane analogue, which is
Proposition \ref{every}.  Therefore, 
we announce and prove it separately as follows. 

\begin{proposition} \label{pevery} For all $n$, the expected number of leaves in a random plane 1-2 tree on vertex set $[n]$ is at least $1/4$.
\end{proposition}

\begin{proof} We prove that the expected number of leaves in plane 1-2 trees on vertex set $[n]$ is at least as large as the expected
number  of leaves in non-plane 1-2 trees on vertex set $[n]$. As the latter has been proved to be at least
$n/4$ in Proposition \ref{every}, this will be sufficient. 

Clearly, in both tree varieties, the number of vertices with two children is one less than the number of leaves. Therefore,
it suffices to prove that average the number of vertices with one child is {\em at most as large} in plane 1-2 trees on $[n]$ as it is on non-plane 1-2 trees on $[n]$. We use a well-known inequality, known as the {\em Chebyshev sum inequality}
or (a special case of) the {\em rearrangement inequality}. 

\begin{proposition} \label{cheb}  Let $r_1\leq r_2\leq \cdots \leq r_u$ and $t_1\geq t_2\geq \cdots \geq t_u$ be nonnegative real numbers. Then the inequality 
\[\frac{r_1+r_2+\cdots +r_u}{u} \geq \frac{r_1t_1+r_2t_2+\cdots +r_ut_u}{t_1+t_2+\cdots +t_u} \]
holds.
\end{proposition}
See \cite{hardy} or \cite{holstermann} for a proof. 

Let us return to the proof of Proposition \ref{pevery}. Consider all $E_n$ non-plane trees on vertex set $[n]$. Let $s_1,s_2,\cdots ,s_{E_n}$ denote the number
of vertices with exactly one child in each of these trees, and let us order the set of these $E_n$ trees so that the sequence of the $s_i$ is non-decreasing, that is,  $s_1\leq s_2\leq \cdots \leq s_{E_n}$. Then the average number of vertices with one child in all
non-plane 1-2 trees on $[n]$ is
\begin{equation} \label{firstav}  M_n = \frac{\sum_{i=1}^{E_n} s_i}{E_n} .\end{equation}
On the other hand, if such a tree $T$ has $s_i$ vertices with one child, then it has $(n-1-s_i)/2$ vertices with two children,
(and, though we will need this only later, $T$ has $(n+1-s_i)/2$ leaves). 
Therefore, there are exactly $2^{(n-1-s_i)/2}$ {\em plane} 1-2 trees on vertex set $[n]$ that are identical to $T$ as
non-plane trees, and each of those trees has $s_i$ vertices with one child. This proves that the average number of 
vertices with exactly one child in all {\em plane} 1-2 trees on vertex set $[n]$ is 
\begin{equation} \label{secondav} m_n =  
 \frac{\sum_{i=1}^{E_n} s_i 2^{(n-1-s_i)/2}}{\sum_{i=1}^{E_n} 2^{(n-1-s_i)/2}} .\end{equation}

Finally, note that the sequences $s_1\leq s_2\leq \cdot \leq s_{E_n}$ and $2^{(n-1-s_1)/2}\geq 2^{(n-1-s_2)/2}\geq 
\cdots \geq 2^{(n-1-s_{E_n})/2}$ satisfy the requirements of Proposition \ref{cheb}, so $M_n \geq m_n$ holds. 
So the average non-plane 1-2 tree has at least as many vertices with one child as the average plane 1-2 tree of the same size. Therefore, the average {\em plane} 1-2 tree has at least as many leaves as the average non-plane 1-2 tree of the same size. The proof of our claim is now immediate, since we saw in Proposition \ref{every} that the average non-plane
tree on $[n]$ has at least $n/4$ leaves. 
\end{proof}

All remaining steps of Theorem \ref{maintheorem} can be carried out without any extra effort, showing that the 
limits $a_k$ exist for all $k$, for the variety of plane 1-2 trees as well. 

\section{Approximations}
Corollary \ref{numerical} makes numerical approximations of $a_k$ possible. As the upper bound provided by (\ref{upperineq}) was obtained by a rather crude estimate, it is reasonable to assume that the lower bound 
in that corollary is a better estimate for $w$ than the upper bound.  It follows from our methods that both the upper and the lower bounds will be
of the form $\pi^{-2}F(\pi)$, where $F$ is a polynomial function with rational coefficients.  For instance, selecting $k=2$
and $r=12$ leads to a lower bound of $ 0.188285\leq a_2$.  On the other hand, less rigorous, but more extensive 
computations carried out by Jay Pantone \cite{pantone} suggest the  approximate values
$a_2 \approx  0.20278137$, 
   $a_3 \approx 0.0893474$, and 
 $ a_4 \approx 0.0243854$.

\begin{center} {\bf Acknowledgement} 
\end{center}

The authors are indebted to Svante Janson for the proof of Theorem \ref{tight}. They are also grateful to Jay Pantone
who helped them obtaining the numerical results of the last section.

\end{document}